\documentclass[reqno]{amsart}
\usepackage{natbib}
\usepackage[pdftex,plainpages=false,colorlinks,hyperindex,bookmarksopen,linkcolor=red,citecolor=blue,urlcolor=blue]{hyperref}

\bibpunct{[}{]}{;}{n}{,}{,}

\newtheorem{te}{Theorem}[section]
\newtheorem{os}[te]{Remark}
\newtheorem{lem}[te]{Lemma}

\numberwithin{equation}{section}

\allowdisplaybreaks

\def \l { \left( }
\def \r {\right) }
\def \ll { \left\lbrace }
\def \rr { \right\rbrace }

\begin{document}

\title[Population models at stochastic times]{Population models at stochastic times}
\author{Enzo Orsingher}
\email{enzo.orsingher@uniroma1.it}
\author{Costantino Ricciuti}
\email{costantino.ricciuti@uniroma1.it}
\author{Bruno Toaldo} 
 \email{bruno.toaldo@uniroma1.it}
\address{Department of Statistical Sciences, Sapienza - University of Rome}
\date{\today}

\begin{abstract}	
In this article, we consider  time-changed models of population evolution $\mathcal{X}^f(t)=\mathcal{X}(H^f(t))$, where $\mathcal{X}$ is a counting  process and $H^f$ is a subordinator with Laplace exponent $f$. In the case $\mathcal{X}$ is a pure birth process, we study  the form of the distribution, the intertimes between successive jumps and the condition of explosion (also in the case of killed subordinators). We also investigate the case where $\mathcal{X}$ represents a death process (linear or sublinear) and study the extinction probabilities as a function of the initial population size $n_0$. Finally, the subordinated linear birth-death process is considered. A special attention is devoted to the case where birth and death rates coincide; the sojourn times are also analysed.
\end{abstract}
	
\maketitle

\section{Introduction}
Birth and death processes can be applied in modelling many dynamical systems, such as cosmic showers, fragmentation processes, queueing systems, epidemics, population growth and aftershocks in earthquakes. The time-changed version of such processes has also been analysed since it is useful to describe the dynamics of various systems when the underlying environmental conditions randomly change.  For example, the fractional birth and death processes, studied in \citet{orspolber3, orspolber4, orspolber1, orspolber2}, are time-changed processes where the distribution of the  time is related to the fractional diffusion equations. On this point consult \citet{cahoy, cahoypol} for some applications and simulations. 

In this paper, we consider the case where the random time is a subordinator. Actually, subordinated Markov processes have been extensively studied since the Fifties. The case of birth and death processes merits however a further investigation and this is the role of the present paper. We consider here compositions of point processes $\mathcal{X}(t)$, $t>0$, with an arbitrary subordinator $H^f(t)$ related to the Bern\v{s}tein functions $f$. We denote such processes as $\mathcal{X}^f(t)= \mathcal{X}(H^f(t))$. The general form of $f$ is as follows
\begin{align} 
f(x)= \alpha+\beta x+\int _0^\infty (1-e^{-xs}) \nu(ds)    \qquad  \alpha \geq 0, \beta \geq 0,
\label{Bernstein}
\end{align}
where $\nu$ is the L\'evy measure satisfying
\begin{align} 
\int_0^\infty (s \wedge 1 ) \nu(ds) < \infty.
\label{misura di Levy}
\end{align}
In this paper we refer to the case $\alpha=\beta=0$, unless explicitly stated.
The structure of the paper is as follows: section 2  treats the subordinated non-linear birth process; section 3 deals with the subordinated linear and sublinear death processes;  section 4 analyses the linear birth-death process, with particular attention to the case where birth and death rates coincide.
In all three cases, we compute directly the state probabilities by means of the composition formula
\begin{align}
\Pr \ll \mathcal{X}^f(t)=k \rr = \int_0^{\infty} \Pr \ll \mathcal{X}(s)=k \rr \Pr \ll H^f(t) \in ds \rr.
\end{align}
Despite most of the subordinators do not possess an explicit form for the probability density function, the distribution of $\mathcal{X}(H^f(t))$ always presents a closed form in terms of the Laplace exponent $f$. We also study the transition probabilities, both for finite and infinitesimal time intervals. We emphasize that the subordinated point processes have a fundamental difference with respect to the classical ones, in that they perform upward or downward jumps of arbitrary size.  For infinitesimal time intervals, we provide a direct and simple proof of the following fact:
\begin{align}
\Pr \ll \mathcal{X}^f(t+dt)=k | \mathcal{X}^f(t)=r \rr = dt \int_0^{\infty} \Pr \ll \mathcal{X}(s)=k | \mathcal{X}(0)=r \rr \nu(ds),
\label{14b}
\end{align}
which is related to Bochner subordination (see \cite{phillips}).\\
The first case taken into account is that of a non-linear birth process with birth rates $\lambda_k$, $k\geq 1$, which is denoted by $\mathcal{N}(t)$.
The subordinated process $\mathcal{N}^f(t)$ does not explode if and only if the following condition is fullfilled
\begin{align}
\sum_{j=1}^\infty \frac{1}{\lambda_j} \, = \, \infty.
\end{align}
This is the same condition of non-explosion holding for the classical case. Such a condition ceases to be true if we consider a  L\'evy exponent with $\alpha \neq 0$, which is related to the so-called killed subordinator.
In this case, indeed, the process $\mathcal{N}^f(t)$ can explode in a finite time, even if $\mathcal{N}(t)$ does not; more precisely
\begin{align}
\Pr \ll \mathcal{N}^f(t)= \infty \rr = 1-e^{-\alpha t}.
\end{align}
We note that $\mathcal{N}^f(t)$ can be regarded as a process where upward jumps are separated by exponentially distribuited time intervals $Y_k$ such that
\begin{align}
\Pr \ll Y_k>t | \mathcal{N}^f(T_{k-1})=r \rr= e^{-f(\lambda_r) t}
\end{align}
where $T_{k-1}$ is the instant of the ($k-1$)-th jump.

In section 3 we study the subordinated linear and sublinear death processes, that we respectively denote by $M^f(t)$ and $\mathbb{M}^f(t)$, with an initial number of components $n_0$. We emphasize that in the  sublinear case the annihilation is initially slower, then accelerates when few survivors remain. So, despite $M^f(t)$ and $\mathbb{M}^f(t)$ present different state probabilities, we observe that the extinction probabilities coincide and we prove that they decrease for increasing values of $n_0$.

In section 4, the subordinated linear birth-death process $L^f(t)$ is considered. If the  birth and death rates coincide and  $H^f$ is a stable subordinator, we compute the mean sojourn time in each state and find, in some particular cases, the distribution of the intertimes between successive jumps. 
We finally study the probability density of the sojourn times, by giving a sketch of the derivation of their Laplace transforms.

\section{Subordinated non-linear birth process}
We consider in this section the process $ \mathcal{N}^f(t)=  \mathcal{N}(H^f(t))$, where $\mathcal{N}$ is a non-linear birth process  with one progenitor and rates $\lambda _k$, $k \geq 1 $, and $H^f(t)$ is a subordinator independent from $\mathcal{N}(t)$. It is well known that the state probabilities of $\mathcal{N}(t)$ read
\begin{align}
\Pr \ll\mathcal{N}(t)=k| \mathcal{N}(0)=1 \rr = \, & \begin{cases} \prod_{j=1}^{k-1} \lambda _ j\sum_{m=1}^k \frac{e^{-\lambda_m t }}{ \prod_{l=1 ,  l \neq m}^k (\lambda_l-\lambda_m )}, \qquad & k>1, \\    
e^{-t f(\lambda_1) }, & k=1.
\end{cases} 
\end{align}
The subordinated  process $\mathcal{N}^f(t)$ thus possesses the following distribution:
\begin{align}
\Pr \ll \mathcal{N}^f(t)=k|\mathcal{N}^f(0)=1 \rr \, = \, & \int_0^\infty  \Pr \ll \mathcal{N}(s)=k | \mathcal{N}(0)=1\rr \Pr \ll H^f(t) \in ds \rr \notag  \\
 \, = \, &  \begin{cases} \prod_{j=1}^{k-1} \lambda _ j\sum_{m=1}^k \frac{e^{-t \, f(\lambda _m )}}{ \prod_{l=1 ,  l \neq m}^k (\lambda_l-\lambda_m )}, \qquad & k>1, \\    
e^{-t f(\lambda_1) }, & k=1.
\end{cases} 
\label{non linear birth 1 progenitor}
\end{align}
The distribution  \eqref{non linear birth 1 progenitor} can be easily generalised to the case of $r$ progenitors  and reads
\begin{align}
\Pr \ll  \mathcal{N}^f(t)=r+k | \mathcal{N}^f(0)=r \rr \, = \,\begin{cases}
\prod_{j=r}^{r+k-1} \lambda _ j\sum_{m=r}^{r+k} \frac{e^{-tf(\lambda _m) }}{ \prod_{l=r ,  l \neq m}^{r+k} (\lambda_l-\lambda_m )}, \quad & k>0, \\ e^{-tf(\lambda_r) }, & k=0.
\end{cases}
\label{23} 
\end{align}
The subordinated process $\mathcal{N}^f(t)$ is time-homogeneous and Markovian. So, the last formula permits us to write
\begin{align}
&\Pr \ll  \mathcal{N}^f(t+dt)=r+k | \mathcal{N}^f(t)=r \rr \notag \\
 = \, & \begin{cases}
\prod_{j=r}^{r+k-1} \lambda _ j\sum_{m=r}^{r+k} \frac{1-dt f(\lambda_m )}{ \prod_{l=r ,  l \neq m}^{r+k} (\lambda_l-\lambda_m )}, \quad & k>0, \\
 1-dt f(\lambda_r), & k=0.
\label{rprog}
\end{cases}
\end{align}
To find an alternative expression for the transition probabilities we need the following
\begin{lem}
For any sequence of $k+1$ distinct positive numbers $\lambda_r, \lambda_{r+1} \cdots \lambda _{r+k}$ the following relationship holds:
\begin{align}
c_{r,k}=\sum_{m=r}^{r+k} \frac{1}{ \prod_{l=r ,  l \neq m}^{r+k} (\lambda_l-\lambda_m )}=0.
\label{Vandermonde}
\end{align}
\end{lem}
\begin{proof}
It is a consequence of \eqref{23} by letting $t \to 0$. An alternative proof can be obtained by suitably adapting the calculation in Theorem 2.1 of \cite{orspolber4}.
\end{proof}
We are now able to state the following theorem.
\begin{te}
For $k>r$ the transition probability takes the form
\begin{align}
\Pr \ll  \mathcal{N}^f(t+dt)=k| \mathcal{N}^f(t)=r \rr = dt\, \int _0 ^{\infty} \Pr \ll \mathcal{N}(s)=k|\mathcal{N}(0)=r \rr \nu (ds) 
\end{align}
\end{te}
\begin{proof}
By  repeatedly using both \eqref{Vandermonde} and the representation \eqref{Bernstein} of the Bern\v{s}tein functions $f$, we have that
\begin{align}
\Pr \ll \,  \mathcal{N}^f(t+dt)=k | \mathcal{N}^f(t)=r \rr \, &= \, \prod_{j=r}^{r+k-1} \lambda_ j\sum_{m=r}^{r+k} \frac{1-dt f(\lambda _m )}{ \prod_{l=r ,  l \neq m}^{r+k} (\lambda_l-\lambda_m )} \notag \\
&= \, -dt\prod_{j=r}^{r+k-1} \lambda_ j\sum_{m=r}^{r+k} \frac{f(\lambda _m )}{ \prod_{l=r ,  l \neq m}^{r+k} (\lambda_l-\lambda_m )}\notag \\
& =- dt \int _0^\infty \prod_{j=r}^{r+k-1} \lambda _ j\sum_{m=r}^{r+k} \frac{1-e^{-\lambda_m s }}{ \prod_{l=r ,  l \neq m}^{r+k} (\lambda_l-\lambda_m )} \nu (ds) \notag \\
 &= dt \int _0^\infty \prod_{j=r}^{r+k-1} \lambda _ j\sum_{m=r}^{r+k} \frac{e^{-\lambda_m s }}{ \prod_{l=r ,  l \neq m}^{r+k} (\lambda_l-\lambda_m )} \nu (ds) \label{ccc}.
\end{align}
In light of \eqref{Vandermonde}, the integrand in \eqref{ccc} is  $ \mathcal{O}( s)$ for $s \to 0$. Reminding \eqref{misura di Levy}, this ensures the convergence of \eqref{ccc}, and the proof is thus complete.
\end{proof}
\begin{os}
For the sake of completeness, we observe that in the case $k=0$ we have
\begin{align}
\Pr \ll \,  \mathcal{N}^f(t+dt)=r | \mathcal{N}^f(t)=r \rr \, = \, & 1-dt f(\lambda _r) \notag \\
= \, & 1-dt \int _0^\infty (1-e^{-\lambda _r s}) \nu(ds)\\
 = \, & 1-dt \int _0^\infty (1-\Pr \ll \mathcal{N}(s)=r | \mathcal{N}(0)=r \rr) \nu(ds).
\end{align}
\end{os}
\begin{os}
The subordinated non-linear birth process performs jumps of arbitrary height as the subordinated Poisson process (see, for example, \citet{orstoa}). Thus, in view of  markovianity, we can write the governing equations for the state probabilities $p_k^f(t)= \Pr \ll \mathcal{N}^f(t)=k | \mathcal{N}^f(0) =1 \rr$. For $k>1$ we have that
\begin{align}
&\frac{d}{dt} p_k^f(t)\, = \, -f(\lambda_k) p_k^f(t)+ \sum_{r=1}^{k-1}  p_r^f(t)   \int _0^\infty \prod_{j=r}^{k-1} \lambda_j \sum_{m=r}^{k} \frac{e^{-\lambda_m s }}{ \prod_{l=r ,  l \neq m}^{k} (\lambda_l-\lambda_m )} \nu (ds), 
\end{align}
while for $k=1$
\begin{align}
&\frac{d}{dt} p_1^f(t) \, = \, - f(\lambda_1)  p_1^f(t).
\end{align} 
\end{os}
\begin{os}
The process $\mathcal{N}(H^f(t))$ presents positive and integer-valued jumps occurring at random times $T_1, T_2, \cdots T_n $. The inter-arrival times $Y_1, Y_2, \cdots Y_n $ are defined as
\begin{align}
Y_k= T_k-T_{k-1}.
\end{align}
It is easy to prove that 
\begin{align}
\Pr \ll Y_k >t |\mathcal{N}^f(T_{k-1})=r \rr= e^{-f(\lambda_r)t}.
\end{align}
This can be justified by considering that in the time intervals $[T_{k-1}, T_{k-1}+t]$, no new offspring appears in the population and thus, by \eqref{rprog}, we have
\begin{align}
\Pr \ll Y_k >t |\mathcal{N}^f(T_{k-1})=r \rr= \Pr \ll \mathcal{N}^f(t+T_{k-1})=r|\mathcal{N}^f(T_{k-1})=r)\rr= e^{-f(\lambda_r)t}.
\end{align}
\end{os}

\subsection{Condition of explosion for the subordinated non-linear birth process}
We note that the explosion of the process $\mathcal{N}^f(t)$, $t>0$, in a finite time is avoided if and only if 
\begin{align}
T_{\infty}= Y_1+Y_2 \cdots Y_{\infty}= \infty
\end{align}
where $Y_j$ , $j \geq 1$, are the intertimes between successive jumps (see \cite{grimmet}, p. 252).
For the non-linear classical process we have that
\begin{align}
\mathbb{E} e^{-T_{\infty}}\, = \, & \mathbb{E}e^{- \sum _{j=1}^{\infty}Y_j}= \lim_{n \to \infty} \prod_{j=1}^{n} \mathbb{E} e^{-Y_j}= \lim_{n \to \infty} \prod _{j=1}^n \frac{\lambda _j}{1+\lambda _j} \notag \\
 = \, & \prod _{j=1}^{\infty} \frac{1}{1+\frac{1}{\lambda _j}}= \frac{1}{1+ \sum _{j=1}^{\infty}\frac{1}{\lambda _j}+ \cdots}.
\end{align}
 So, if $\sum_{j=1}^\infty \frac{1}{\lambda_j}= \infty$ we have $e^{-T_{\infty}}=0$ a.s., that is   $T_{\infty}=\infty$.
Therefore, for the subordinated non-linear birth process we have that
\begin{align}
\Pr \ll \mathcal{N}^f(t) < \infty \rr \, = \, & \int _0 ^{\infty} \sum _{k=1}^{\infty} \Pr \ll \mathcal{N}(s)=k \rr \Pr\ll {H^f(t) \in ds}\rr \notag \\
= \, & \int_0^{\infty} \Pr\ll {H^f(t) \in ds}\rr =1, \qquad \forall t>0.
\end{align}
Instead, if $\sum _{j=1}^{\infty} \frac{1}{\lambda _j} <\infty$, we get $\sum _{k=1}^{\infty}  \Pr \ll \mathcal{N}(s)=k \rr < \infty $, and this implies that $\Pr \ll \mathcal{N}^f(t)<\infty \rr <1$.

We can now consider the case of killed subordinators $\mathcal{H}^g(t)$, defined as
\begin{align}
\mathcal{H}^g(t)= \begin{cases}
H^f(t), &\qquad t< T,\\
\infty, &\qquad t \geq T,
\end{cases}
\end{align}
where $T\sim Exp(\alpha)$ and $H^f(t)$ is an ordinary subordinator related to the function $f(x)= \int _0^{\infty}(1-e^{-sx}) \nu(ds)$. It is well-known that $\mathcal{H}^g(t)$ is related to a Bern\v{s}tein function
\begin{align}
g(x)= \alpha+f(x).
\end{align}
In this case, even if $\sum_{j=1}^\infty \frac{1}{\lambda_j}= \infty$ , the probability of explosion for $\mathcal{N}^f(t)$ is positive and equal to 
\begin{align}
\Pr \ll \mathcal{N}^f(t)= \infty \rr = 1- e^{-t \alpha}.
\end{align}
This can be proven by observing that
\begin{align}
\Pr \ll \mathcal{N}^f(t) < \infty \rr \, = \, & \int _0 ^{\infty} \sum _{k=1}^{\infty} \Pr \ll \mathcal{N}(s)=k \rr \Pr\ll {H^f(t) \in ds}\rr \notag \\
= \, & \int _0^{\infty} \Pr\ll {H^f(t) \in ds}\rr = \int _0^{\infty}  e^{- \mu s} \Pr\ll {H^f(t) \in ds}\rr\bigg|_{\mu =0} \notag \\
 = \, & e^{- \alpha t- f(\mu) t}\bigg|_{\mu =0}= e^{-\alpha t}.
\end{align}
If, instead, $\sum_{j=1}^\infty \frac{1}{\lambda_j}< \infty$, we have $\sum _{k=1}^{\infty} \Pr \ll \mathcal{N}(s)=k \rr <1$ and, a fortiori, $\Pr \ll \mathcal{N}^f(t)<\infty \rr < e ^{-\alpha t}$.

\subsection{Subordinated linear birth process}
The subordinated Yule-Furry  process $N^f(t)$ with one initial progenitor possesses the following distribution
\begin{align}
p_k^f(t) \, = \, & \int _0 ^ {\infty} e^{- \lambda s } (1-e^{- \lambda s })^{k-1} \Pr \lbrace  H^ {f}(t) \in ds  \rbrace \notag \\
= \, & \int _0 ^ {\infty} e^{- \lambda s }\sum _{J=0}^{k-1} \binom{k-1}{j}(-1)^j e^{- \lambda sj }\Pr \lbrace  H^ {f}(t) \in ds  \rbrace\notag \\
= \, & \sum _{j=0}^{k-1}\binom{k-1}{j} (-1)^j \int _0 ^{\infty} e^{- s(\lambda + \lambda j) } \Pr \lbrace  H^ {f}(t) \in ds  \rbrace \notag \\
= \, & \sum _{j=0}^{k-1}\binom{k-1}{j} (-1)^j e^{-t\, f(\lambda (j+1))}.
\end{align}
Of course, this is obtainable from the distribution $\mathcal{N}^f(t)$ by assuming that $\lambda _j=\lambda j$
We now compute the factorial moments of the subordinated linear birth process.
The probability generating function is 
\begin{align}
G^f(u,t)= \sum_{k=1}^\infty u^k \int_0^\infty e^{-\lambda s} (1-e^{-\lambda s})^{k-1} \Pr (H^f(t) \in ds).
\end{align}
The $r$-th order factorial moments are
\begin{align}
& \frac{\partial^r}{\partial u^r} G^f(u,t)\bigg|_{u=1} \notag \\
= \, & \sum_{k=r}^\infty k(k-1)\cdots(k-r+1) \int_0^\infty e^{-\lambda s} (1-e^{-\lambda s})^{k-1} \Pr \ll H^f(t) \in ds\rr \notag \\
= \, & \sum_{k=r}^\infty k(k-1)\cdots(k-r+1) \int_0^\infty e^{-\lambda s} (1-e^{-\lambda s})^{k-r} (1-e^{-\lambda s})^{r-1}\Pr \ll H^f(t) \in ds \rr
\end{align}
and since 
\begin{align}
\sum_{k=r}^\infty k(k-1)...(k-r+1)(1-p)^{k-r}
= (-1)^r\frac{d^r}{dp^r} \sum_{k=0}^\infty (1-p)^k
= (-1)^r\frac{d^r}{dp^r} \frac{1}{p}= \frac{r!}{p^{r+1}}
\end{align}
we have that
\begin{align}
\frac{\partial^r}{\partial u^r} G(u,t)\bigg|_{u=1} \, = \, & r!\int_0^\infty e^{\lambda r s }(1-e^{-\lambda s})^{r-1} \Pr \ll H^f(t) \in ds \rr \notag  \\
= \, & r! \sum_{m=0} ^{r-1} \begin{pmatrix}
r-1 \\ m
\end{pmatrix}(-1)^m \int_0^\infty e^{- \lambda s(m-r)} \Pr \ll H^f(t) \in ds\rr \\
 = \, & r! \sum_{m=0} ^{r-1} \begin{pmatrix}
r-1 \\ m
\end{pmatrix}(-1)^m  e^{-t f(\lambda (m-r))}.
\end{align}
By $f(-x)$, $x>0$ we mean the extended Bern\v{s}tein function, having representation
\begin{align} 
f(-x)=\int_0^{\infty} (1-e^{sx}) \nu(ds), \qquad x>0,
\label{extended bernstein}
\end{align}
provided that the integral in \eqref{extended bernstein} is convergent.
In particular, we infer that
\begin{align}
\mathbb{E}(\mathcal{N}^f(t))=e^{-tf(-\lambda)}
\end{align}
and
\begin{align}
\textrm{Var} ( \mathcal{N}^f(t) )= 2e^{-tf(-2\lambda)}- e^{-tf(-\lambda)}-e^{-2tf(-\lambda)}.
\end{align}
For  a stable subordinator, that is with L\'evy measure  $\nu(ds)= \frac{\alpha s^{-\alpha -1}}{\Gamma(1-\alpha)}ds$ , $\alpha \in (0,1)$,  all the factorial moments are infinite.
Instead, for a tempered stable subordinator, where $\nu(ds)= \frac{\alpha e^{-\theta s} s^{-\alpha -1}}{\Gamma(1-\alpha)}ds$, $\alpha \in (0,1)$ and $\theta>0$,  only the factorial moments of order $r$ such that $r< \frac{\theta}{\lambda}$ are finite. 
If we then consider the Gamma subordinator, with $\nu(ds)= \frac{e^{-\alpha s }}{s}ds$, only the factorial moments of order $r$ such that $r< \frac{\alpha}{\lambda}$ are finite.

\subsection{Fractional subordinated non-linear birth process}
The fractional non-linear birth process has state probabilities $ p_k^{\nu}(t)$ solving the fractional differential equation
\begin{align}
\frac{d^{\nu}p_k^{\nu}(t)}{dt^{\nu}}= -\lambda _k p_k^{\nu}(t) +\lambda _{k-1}p_{k-1}^{\nu}(t) \qquad \nu \in (0,1), k \geq 1
\end{align}
with initial condition
\begin{align}
p_k^{\nu}(0)= \begin{cases} 1, \qquad &k=1, \\ 0, & k>1. \end{cases}
\end{align}
The state probabilities read (see Orsingher and Polito \cite{orspolber1})
\begin{align} 
   p_k^{\nu}(t)= \Pr \ll \mathcal{N}^{\nu}(t)=k | \mathcal{N}^{\nu}(0)=1 \rr= \prod_{j=1}^{k-1} \lambda_j \sum_{m=1}^{k} \frac{E_{\nu,1}(-\lambda_m t^{\nu} )}{ \prod_{l=1 ,  l \neq m}^{k} (\lambda_l-\lambda_m )}  \qquad \nu \in (0,1),
\end{align}
where
\begin{align}
E_{\nu,1}(-\eta t^{\nu})= \frac{\sin (\nu \pi) }{\pi} \int _0 ^{\infty}  \frac{r^{\nu -1 } e^{-r \eta ^{\frac{1}{\nu}}t}}{r^{2 \nu}+2r^{\nu}\cos(\nu \pi)+1}dr
\end{align}
is the Mittag-Leffler function (see formula (7.3) in \citet{saxena}).
So, the subordinated non-linear fractional birth process has distribution 
\begin{align}
& \Pr \ll \mathcal{N}^{\nu}(H^f(t))=k | \mathcal{N}^{\nu}(0)=1 \rr \notag  \\
= \, & \prod_{j=1}^{k-1} \lambda_j \sum_{m=1}^{k} \frac{1}{ \prod_{l=1 ,  l \neq m}^{k} (\lambda_l-\lambda_m )}\frac{\sin (\nu \pi) }{\pi} \int _0 ^{\infty}  \frac{r^{\nu -1 } e^{-tf(r \lambda_m ^{\frac{1}{\nu}})}}{r^{2 \nu}+2r^{\nu}\cos(\nu \pi)+1}dr.
\end{align}

\section{Subordinated  death processes}
We now consider the process $M^f(t)= M(H^f(t))$, where $M$ is a linear death process with $n_0$ progenitors. The state probabilities read
\begin{align}
&\Pr \ll  M^f(t)=k | M^f(0) = n_0 \rr  =  \int _0^{\infty} \binom{n_0}{k} e^{-\mu ks}(1-e^{-\mu s}) ^{n_0-k} \Pr \ll  H^f(t) \in ds \rr  \notag \\ 
 = \, & \begin{pmatrix}
n_0 \\ k
\end{pmatrix} \sum _ {j=0} ^ {n_0-k}\begin{pmatrix}
n_0-k \\ j
\end{pmatrix}(-1)^j \int _0 ^  {\infty}e^{- (\mu k + \mu j)s }\Pr \ll  H^f(t) \in ds \rr \notag \\
= \, &  \begin{pmatrix}
n_0 \\ k
\end{pmatrix} \sum _ {j=0} ^ {n_0-k}\begin{pmatrix}
n_0-k \\ j
\end{pmatrix}(-1)^j e^ {-tf(\mu k + \mu j)}.
\end{align}
In particular, the extinction probability is
\begin{align}
\Pr \ll  M^f(t)=0 | M^f(0) = n_0 \rr  \,= & \sum _ {j=0} ^ {n_0}\begin{pmatrix}
n_0 \\ j
\end{pmatrix}(-1)^j e^ {-tf(\mu j)}\notag \\
= \, & 1+ \sum _ {j=1} ^ {n_0}\begin{pmatrix}
n_0 \\ j
\end{pmatrix}(-1)^j e^ {-tf(\mu j)}
\end{align}
and converges to $1$ exponentially fast with rate $f(\mu)$.
\begin{os}
We observe that the extinction probability is a decreasing function of $n_0$ for any choice of the subordinator $H^f(t)$. This can be shown by observing that
\begin{align}
&\Pr \ll M^f(t)=0 |M^f(0)=n_0 \rr -\Pr \ll M^f(t)=0 |M^f(0)=n_0-1 \rr \notag  \\
= \, & \sum _ {j=1} ^ {n_0}\begin{pmatrix}
n_0 \\ j
\end{pmatrix}(-1)^j e^ {-tf(\mu j)} -\sum _ {j=1} ^ {n_0-1}\begin{pmatrix}
n_0-1 \\ j
\end{pmatrix}(-1)^j e^ {-tf(\mu j)} \notag \\
= \, & \sum _ {j=1} ^ {n_0-1}\begin{pmatrix}
n_0-1 \\ j-1
\end{pmatrix}(-1)^j e^ {-tf(\mu j)}+(-1)^{n_0}e^{-tf(\mu n_0)}\notag \\
= \, &  \sum _ {j=1} ^ {n_0}\begin{pmatrix}
n_0-1 \\ j-1
\end{pmatrix}(-1)^j e^ {-tf(\mu j)} \notag \\
= \, & -\sum _ {j=0} ^ {n_0-1}\begin{pmatrix}
n_0-1 \\ j
\end{pmatrix}(-1)^j e^ {-tf(\mu (j+1)} \notag \\
= \, &-\int_0 ^{\infty} \sum _ {j=0} ^ {n_0-1}\begin{pmatrix}
n_0-1 \\ j
\end{pmatrix}(-1)^j e^ {-s\mu (j+1)} \Pr \ll H^f(t) \in ds \rr \notag \\
= \, &-\int_0^\infty e^{-\mu s}(1-e^{-\mu s})^{n_0-1}\Pr \ll H^f(t) \in ds \rr 
<  0.
\end{align}
This permits us also to establish the following upper bound which is valid for all values of $n_0$.
\begin{align}
\Pr \ll M^f(t)=0 |M^f(0)=n_0 \rr < \Pr \ll M^f(t)=0 |M^f(0)=1 \rr = 1-e^{-tf(\mu)}.
\end{align}
We also infer that
\begin{align*}
& \Pr \ll M^f(t)= k|M^f(0)=n_0 \rr = \\ & \Pr \ll M^f(t)=k|M^f(0)=n_0 -1 \rr - \frac{1}{n_0} \Pr \ll M^f(t)=1|M^f(0)=n_0 \rr  \qquad \forall k< n_0
\end{align*}
\end{os}
\begin{os}
The probability generating function of the subordinated linear death process is 
\begin{align}
G(u,t)= \int_0^\infty (ue^{- \mu s}+1-e^{- \mu s})^{n_0} \Pr \ll H^f(t)  \in ds \rr.
\end{align}
We now compute the factorial moments of order $r$ for the process $M^f(t)$:
\begin{align}
&\mathbb{E} \bigl ( M^f(t) (M^f(t)-1)(M^f(t)-2) \cdots (M^f(t)-r+1) \bigr ) \notag  \\
= \, &\int_0^\infty \frac{\partial^r}{\partial u^r} (ue^{- \mu s}+1-e^{- \mu s})^{n_0}|_{u=1} \Pr \ll H^f(t)  \in ds \rr \notag \\
  = \, & n_0(n_0-1)(n_0-2)...(n_0-r+1) \int_0^\infty e^ {-\mu r s } \Pr \ll H^f(t)  \in ds \rr \notag  \\
  = \, & n_0(n_0-1)(n_0-2)...(n_0-r+1) e^{-tf(\mu r )} \notag \\
= \, & r! \binom{n_0}{r}  e^{-tf(\mu r )} \qquad \qquad \textrm{for } r \leq n_0.
\end{align}
In particular, we extract the expressions 
\begin{align}
\mathbb{E} \, M^f(t)= n_0 e^ {-t \, f(\mu)} 
\end{align}
and 
\begin{align}
\textrm{Var} \,M^f(t)= n_0 e^{-tf(\mu)}-n_0 e^{-tf(2 \mu)} + n_0^2 e^{-t f(2 \mu)} -n_0 ^2 e^{-2t f(\mu)}.
\end{align}
The variance can be also be obtained as
\begin{align}
\textrm{Var} \, M^f(t)= \, & \mathbb{E} \ll  \textrm{Var} \, (M(H^f(t))|H^f(t)) \rr + \textrm{Var} \,  \ll \mathbb{E}(M(H^f(t))|H^f(t)) \rr \notag \\
= \, & \mathbb{E}\bigl ( n_0 e^{-\mu H^f(t)}(1-e^{-\mu H^f(t)})   \bigr )+ \textrm{Var}\, ( n_0 e ^{-\mu H^f(t)})\notag \\
= \, & n_0 e^{-tf(\mu)}-n_0 e^{-tf(2 \mu)} + n_0^2 e^{-t f(2 \mu)} -n_0 ^2 e^{-2t f(\mu)}.
\end{align} 
\end{os}
\begin{os}
The transition probabilities
\begin{align}
\Pr \ll  M^f(t_0+t)=k|M^f(t_0)=r \rr   =   \binom{r}{k} \sum_{j=0}^{r-k}\binom{r-k}{j} (-1)^j e^{-tf(\mu k + \mu j)}
\end{align}
permit us to write, for a small time interval $[t,t+dt)$,
\begin{align}
& \Pr \ll  M^f(t_0+dt)=k |   M^f(t_0)=r \rr \notag \\
= \, & \binom{r}{k} \sum_{j=0}^{r-k}\binom{r-k}{j}(-1)^j (1-dt \,f(\mu k + \mu j) ) \notag \\
= \, & -dt \binom{r}{k} \sum_{j=0}^{r-k}\binom{r-k}{j}(-1)^j  \, f(\mu k + \mu j) \notag \\
= \, & -dt \binom{r}{k} \sum_{j=0}^{r-k}\binom{r-k}{j}(-1)^j  \, \int _0 ^\infty (1-e^{-(\mu k + \mu j)s}) \nu (ds) \notag \\
 = \, & dt \binom{r}{k} \int _0 ^\infty \sum _ {j=0} ^ {r-k}\binom{r-k}{j}(-1)^j  \,  e^{- \mu js} e^{-\mu k s} \nu (ds) \notag \\
= \, & dt \int _0 ^\infty \binom{r}{k} (1-e^{- \mu s})^{r-k} e^{-\mu k s} \nu (ds)\notag \\
= \, & dt \int_0^\infty \Pr \ll M(s) = k | M(0) = r \rr \nu(ds)
\label{trans prob death process} \qquad 0 \leq k <r \leq n_0
\end{align}
It follows that the subordinated death process decreases with downwards jumps of arbitrary size. Formula \eqref{trans prob death process} is a special case of \eqref{14b} for the linear death process.
\end{os}
\begin{os}
If $M^f(t_0)=r$, the probability that the number of individuals does not change during a time interval of length $t$ is
\begin{align}
\Pr \ll  M^f(t_0+t)=r|   M^f(t_0)=r \rr  \, = e^{-tf(r \mu)}.
\label{312}
\end{align}
As a consequence, the random time between two successive jumps has exponential distribution with rate $f(\mu r)$, i.e.
\begin{align}
T_r \sim \textrm{Exp}(f(\mu r)).
\end{align}
From \eqref{312} we have also that
\begin{align}
\Pr \ll M^f(t+dt) = r | M^f(t)  = r \rr \, = \, 1-dtf(\mu r).
\end{align}
\end{os}
\begin{os}
In view of \eqref{trans prob death process} we  can write the governing equations for the transition probabilities $p_k^f(t)= \Pr \ll M^f(t)=k|M^f(0)=n_0 \rr$, for $0 \leq k \leq n_0$
\begin{align}
\frac{d}{dt} p_k^f(t)= -p_k^f(t)f(\mu k)+ \sum_{j=k+1}^{n_0} p_j^f(t)  \int _0 ^\infty \binom{j}{k} (1-e^{- \mu s})^{j-k} e^{-\mu k s} \nu (ds) .
\end{align}
\end{os}
\subsection{The subordinated sublinear death process}
In the sublinear death process we have that, for $0 \leq k \leq n_0$,
\begin{align}
\Pr \ll \mathbb{M}(t+dt)=k-1| \mathbb{M}(t)=k, \mathbb{M}(0)=n_{0}\rr= \mu (n_0-k+1)dt+o(dt)  
\end{align}
so that the probability that a particle disappears in $[t,t+dt)$ is proportional to the number of deaths occurred in $[0,t)$. It is well-known that
\begin{align}
\Pr \ll \mathbb{M}(t)=k| \mathbb{M}(0)=n_0\rr \, = \, \begin{cases} e^{-\mu t}(1-e^{-\mu t })^{n_0-k}, \qquad &k=1,2, \dots, n_0, \\
(1-e^{-\mu t})^{n_0} , & k=0.
\end{cases}
\end{align}
So, the probability law of the subordinated process immediately follows
\begin{align}
 &\Pr \ll \mathbb{M}^f(t)=k| \mathbb{M}^f(0)=n_0\rr \notag \\
  = \, & \begin{cases} \sum_{j=0}^{n_0-k} \begin{pmatrix}
n_0-k \\ j
\end{pmatrix} (-1)^j e^{- tf(\mu (j+1))},   \qquad &k=0,1, \dots, n_0, \\
\sum _{k=0}^{n_0} \begin{pmatrix} 
n_0 \\k
\end{pmatrix} (-1)^k e^{-t f( \mu k) },  &k=0
\end{cases}
\end{align}
The extinction probability is a decreasing function of $n_0$ as in the sublinear death process. Furthermore we observe that the extinction probabilities for the subordinated linear and sublinear death process coincide.

\section{Subordinated linear birth-death processes}
In this section we consider the linear birth and death process $L(t)$ with one progenitor at the time $H^f(t)$.
We recall that, for $k\geq 1$ (see \citet{bailey}, page 90),
\begin{align}
\Pr \ll L(t)=k | L(0) =1 \rr = \begin{cases}
\frac{(\lambda-\mu)^2e^{-(\lambda-\mu)t}(\lambda(1-e^{-(\lambda-\mu)t}))^{k-1}}{(\lambda-\mu e^{-(\lambda-\mu)t})^{k+1}}, \qquad & \lambda > \mu, \\
\frac{(\mu-\lambda)^2 e^{-(\mu-\lambda)t}(\lambda(1-e^{-(\mu-\lambda)t}))^{k-1}}{(\mu-\lambda  e^{-(\mu-\lambda)t})^{k+1}},  & \lambda < \mu, \\
\frac{(\lambda t)^{k-1}}{(1+\lambda t)^{k+1}},  & \lambda = \mu.
\end{cases} 
\end{align}
while the extinction probabilities have the form
\begin{align}
\Pr \ll L(t)=0 |L(0) = 1 \rr = \begin{cases}
\frac{\mu-\mu e^{-t(\lambda-\mu)}}{\lambda-\mu e ^{-t(\lambda- \mu)}}, \qquad & \lambda > \mu,\\
\frac{\mu-\mu e^{-t(\mu- \lambda)}}{\lambda-\mu e ^{-t(\mu- \lambda)}}, & \mu > \lambda, \\
\frac{\lambda t}{1+\lambda t},  & \lambda= \mu.
\end{cases}
\end{align}
We now study the subordinated process $L^f(t)=L(H^f(t))$. When $\lambda \neq \mu$, after  a series expansion we easily obtain that
\begin{align}
&\Pr \ll L^f(t)=k |L^f(0) =1 \rr \notag \\
= \, & 
\begin{cases}
 \l \frac{\lambda - \mu }{\lambda} \r^2 \sum_{l=0}^\infty \binom{l+k}{l} \l \frac{\mu }{\lambda} \r^l \sum_{r=0}^{k-1} (-1)^r \binom{k-1}{r} e^{-t f \l \l \lambda - \mu \r \l l+r+1 \r \r} , \qquad &\lambda >\mu, \\
 \l \frac{\mu -\lambda}{\mu} \r^2 \l \frac{\lambda}{\mu} \r^{k-1} \sum_{l=0}^\infty \binom{l+k}{l} \l \frac{\lambda }{\mu} \r^l \sum_{r=0}^{k-1} (-1)^r \binom{k-1}{r} e^{-t f\l \l \mu - \lambda \r \l l+r+1 \r \r}, & \lambda <\mu,
\end{cases}
\end{align}
provided that $k \geq 1$. Moreover, the extinction probabilities have the following form 
\begin{align}
\Pr \ll L^f(t)=0 \rr \, = \,
\begin{cases}
\frac{\mu - \lambda}{\lambda} \l \sum_{m=1}^\infty \l \frac{\mu}{\lambda} \r^m e^{-tf\l (\lambda - \mu)m \r} \r+\frac{\mu}{\lambda}, \qquad &\lambda>\mu, \\
   1- \l \frac{\mu -\lambda}{\lambda}  \r \sum_{m=1}^\infty \l \frac{\lambda}{\mu} \r^m e^{-tf \l \l \mu - \lambda \r m \r},  &\lambda <\mu.
\end{cases}
\end{align}
Similarly to the classical process, we have
\begin{align}
\lim_{t \to \infty} \Pr \ll L^f(t)=0 \rr = \begin{cases} \frac{\mu}{\lambda}, \qquad &\lambda > \mu, \\
1, & \lambda < \mu.
\end{cases}
\end{align}
\subsection{Processes with equal birth and death rates}
We concentrate ourselves on the case $\lambda= \mu$, which leads to some interesting results. The extinction probability reads
 \begin{align}
\Pr \ll L^f(t) = 0 | L^f(0) =1 \rr \, = \, & \int_0^\infty \frac{\lambda s}{1+\lambda s} \Pr \ll H^f(t) \in ds \rr \notag \\
= \, &1- \int_0^\infty \frac{1}{1+\lambda s} \Pr \ll H^f(t) \in ds \rr \notag \\
= \, & 1-\int_0^\infty \Pr \ll H^f(t) \in ds \rr \int_0^\infty dw \, e^{-w\lambda s} \, e^{-w} \notag \\
= \, &1-\int_0^\infty dw \, e^{-w} e^{-tf(\lambda w)}
\label{exctinction probability}.
\end{align}
We note that
\begin{align}
\lim_{t \to \infty} \Pr \ll L^f(t)=0 |L^f(0) = 1 \rr = 1
\end{align}
as in the classical case. From \eqref{exctinction probability} we infer that the distribution of the extinction time $T_0^f = \inf \ll t \geq 0 : L^f(t) = 0 \rr$, has the following form
\begin{align}
\Pr \ll T_0^f \in dt \rr / dt \, = \, \int_0^\infty e^{-w} f(\lambda w) e^{-tf(\lambda w)} dw.
\end{align}
We now observe that all the state probabilities of the process $L(t)$ depend on the extinction probability (see \cite{orspolber1}) 
\begin{align}
\Pr \ll L(t) = k |L(0) =1 \rr \, = \, & \frac{(\lambda t)^{k-1}}{\l 1+\lambda t \r^{k+1}} \qquad \qquad \qquad  \qquad \qquad \qquad \qquad k\geq 1 \notag \\
= \, & \frac{(-1)^{k-1} \lambda^{k-1}}{k!} \frac{d^k}{d\lambda^k} \l \frac{\lambda}{1+\lambda t} \r \notag \\
= \, & \frac{(-1)^{k-1} \lambda^{k-1}}{k!} \frac{d^k}{d\lambda^k} \l \lambda \l 1-\Pr \ll L(t)=0\rr \r \r.
\label{state prob for L(t)}
\end{align}
Hence, the state probabilities of $L^f(t)$ can be written, for $k \geq 1$, as
\begin{align}
&\Pr \ll L^f(t) = k |L^f(0) =1 \rr \, \notag \\
= \,  &\frac{(-1)^{k-1} \lambda^{k-1}}{k!} \frac{d^k}{d\lambda^k}\left[ \lambda \int_0^\infty    \l 1-\Pr \ll L(s)=0 \rr \r  \Pr \ll H^f(t) \in ds \rr \right] \notag \\
= \, & \frac{(-1)^{k-1} \lambda^{k-1}}{k!} \frac{d^k}{d\lambda^k} \left[ \lambda  \l 1- \Pr \ll L^f(t) = 0 \rr \r \right] \notag \\
=\, &\frac{(-1)^{k-1} \lambda^{k-1}}{k!} \frac{d^k}{d\lambda^k} \left[ \lambda \int_0^\infty dw \, e^{-w} e^{-tf(\lambda w)} \right].
\label{state probabilities L^f(t)}
\end{align}

\subsection{Transition probabilities}
To compute the transition probabilities of $L^f(t)$, we recall that the linear birth-death process with $r$ progenitors has the following probability law (see \cite{bailey}, page 94, formula 8.47):
\begin{align}
\Pr \ll L(t)= n| L(0)=r \rr = \sum _{j=0}^{min(r,n)}   \binom{r}{j} \binom{r+n-j-1}{r-1} \alpha ^{r-j} \beta ^{n-j}(1-\alpha-\beta)^{j}, 
\label{transition prob birth death}
\end{align}
where $n \geq 0$ and
\begin{align}
\alpha \,= \, \frac{\mu(e^{(\lambda-\mu)t}-1)}{\lambda e^{(\lambda-\mu)t}-\mu } \qquad \textrm{and} \qquad \beta \,= \, \frac{\lambda(e^{(\lambda-\mu)t}-1)}{\lambda e^{(\lambda-\mu)t}-\mu }.
\end{align}
In the case $\lambda= \mu$ we have
\begin{align}
\lim_{\mu \to \lambda} \alpha \, = \, \lim_{\mu \to \lambda} \beta = \frac{\lambda t}{1+\lambda t}
\end{align} 
so that
\begin{align}
&\Pr \ll L(t)=n|L(0)=r \rr \notag \\
 = \, & \sum _{j=0}^{min(r,n)}   \binom{r}{j} \binom{r+n-j-1}{r-1} \biggl ( \frac{\lambda t}{1+\lambda t}\biggr )^{r+n-2j}\biggl (1-2 \frac{\lambda t}{1+ \lambda t}\biggr)^{j} \notag \\
= \, & \sum _{j=0}^{min(r,n)}  \sum_{k=0}^j \binom{r}{j} \binom{r+n-j-1}{r-1} \binom{j}{k}(-2)^k \biggl ( \frac{\lambda t}{1+\lambda t}\biggr )^{r+n-2j+k}. \label{mi serve per phillips nascita e morte 2}
\end{align}
One can check that for $r=1$ the last formula reduces to 
\begin{align}
\Pr \ll  L(t)=n |L(0) =1 \rr= \frac{(\lambda t)^{n-1}}{(1+\lambda t)^{n+1}}.
\end{align}
The transition probabilities related to the subordinated process $L^f(t)$ can be written in an elegant form, as shown in the following theorem.
\begin{te}
In the subordinated linear birth-death process $L^f(t)$, when $\lambda=\mu$, $n\geq 0$, $r\geq 1$, $n \neq r$, we have that
\begin{align}
&\Pr \ll L^f(t+t_0)=n |L^f(t_0)=r \rr \notag \\= \,  &\sum _{j=0}^{min(r,n)}  \sum_{k=0}^j \binom{r}{j} \binom{r+n-j-1}{r-1} \binom{j}{k}2^k  \frac{ (-1)^{r+n-1} \lambda ^ {r+n+k-2j}}{(r+n-2j+k-1)!} \notag \\
& \times  \frac{d^{r+n-2j+k-1}}{d \lambda ^{r+n-2j+k-1}} \biggl [ \frac{1}{\lambda}- \frac{1}{\lambda}\int_0^\infty dw \, e^{-w} e^{-tf(\lambda w)} \biggr ] \label{probabilità di transizione processo nascita morte subordinato}
\end{align}
\end{te}
\begin{proof}
By subordination we have
\begin{align}
 \Pr  \ll L^f(t)=n | L^f(0)=r \rr \, = \, & \int _0^{\infty} \Pr \ll L(s)=n |L(0)=r \rr \Pr \ll H^f(t) \in ds \rr \notag  \\ 
 = \, & \sum _{j=0}^{min(r,n)}  \sum_{k=0}^j \binom{r}{j} \binom{r+n-j-1}{r-1} \binom{j}{k}(-2)^k  \notag \\
 & \times \int_0 ^ {\infty} \Pr \ll H(t) \in ds \rr \biggl ( \frac{\lambda s}{1+\lambda s}\biggr )^{r+n-2j+k}.
\end{align}
To compute the last integral, we preliminarly observe that
\begin{align}
\frac{d^m}{d \lambda ^m} \frac{1}{1+\lambda s}= (-1)^m m!\, s^m \frac{1}{(1+\lambda s)^{m+1}}
\end{align}
and consequently
\begin{align}
\biggl ( \frac{\lambda s}{1+ \lambda s } \biggr )^m =  \frac{(-1)^{m-1} s \, \lambda ^m}{(m-1)!}\frac{d^{m-1}}{d \lambda ^{m-1}} \frac{1}{1+\lambda s}. \label{mi serve per phillips in nascita e morte 1}
\end{align}
So, we have
\begin{align}
&\Pr \ll L^f(t)=n|L^f(0)=r \rr \notag \\
 = \, & \sum _{j=0}^{min(r,n)}  \sum_{k=0}^j \binom{r}{j} \binom{r+n-j-1}{r-1} \binom{j}{k}2^k  \frac{ (-1)^{r+n-1} \lambda^{r+n-2j+k}}{(r+n-2j+k-1)!} \notag \\
& \times \frac{d^{r+n-2j+k-1}}{d \lambda ^{r+n-2j+k-1}} \int _0 ^{\infty} \frac{s}{1+\lambda s} \Pr \ll H^f(t) \in ds \rr
\end{align}
where, by using \eqref{exctinction probability}, we write
\begin{align}
 \int _0 ^{\infty} \frac{s}{1+\lambda s} \Pr \ll H^f(t) \in ds \rr &= \frac{1}{\lambda}  \int _0 ^{\infty} \frac{\lambda s}{1+\lambda s} \Pr \ll H^f(t) \in ds \rr \notag \\
 &= \frac{1}{\lambda} \biggl [1-\int_0^\infty dw \, e^{-w} e^{-tf(\lambda w)} \biggr ]
\end{align}
and the desired result immediately follows.
\end{proof}
\begin{os}
For a small time interval $dt$, the quantity in square brackets in (\ref{probabilità di transizione processo nascita morte subordinato}) can be written as
\begin{align*}
& \frac{1}{\lambda}- \frac{1}{\lambda} \int_0^{\infty} dw \, e^{-w}(1-dt f(\lambda w)) \\
& =dt \, \frac{1}{\lambda} \int_0^{\infty} dw \, e^{-w} \int_0^{\infty} \nu(ds) (1-e^{-\lambda w s})\\
& = dt \int_0^{\infty} \nu(ds) \frac{s}{1+\lambda s}
\end{align*}
Then, by using (\ref{mi serve per phillips in nascita e morte 1}) e (\ref{mi serve per phillips nascita e morte 2}), formula (\ref{probabilità di transizione processo nascita morte subordinato}) reduces to
\begin{align*}
\Pr \ll L^f(t_0+dt)=n| L^f(t_0)=k \rr = dt\int_0^{\infty} \nu(ds) \Pr \ll L(s)=n | L(0)=k \rr
\end{align*}
thus proving relation (\ref{14b}) for subordinated birth-death processes.
\end{os}
\begin{os}
If $L^f(0)=1$, from (\ref{state probabilities L^f(t)}) we have that the probability that the number of individuals does not change during a time interval of length $dt$ is
\begin{align*}
\Pr \ll L^f(dt)=1|L^f(0)=1 \rr= 1-dt\, \frac{d}{d \lambda} \bigl (\lambda\int _0^{\infty}dw \, e^{-w}f(\lambda w) \bigr )
\end{align*}
Thus the waiting time for the first jump, i.e.
\begin{align*}
T_1= inf \ll t>0: L^f(t) \neq 1 \rr ,
\end{align*}
has the following distribution
\begin{align}
\Pr \ll T_1>t \rr = e^{-t \frac{d}{d \lambda} (\lambda\int _0^{\infty}dw \, e^{-w}f(\lambda w)}).
\end{align}
For example, in the case $H^f(t)$ is a stable subordinator with index $\alpha \in (0,1)$, $T_1$ has an exponential distribution with parameter $ \lambda ^{\alpha} \Gamma (\alpha +2)$.
\end{os}

\subsection{Mean sojourn times} 
Let $V_k(t)$, $k\geq 1$ the total amount of time that the process $L(t)$ spends in the state $k$ up to time $t$, i.e.
\begin{align}
V_k(t)= \int_0 ^t I_k(L(s))\, ds,
\end{align}
where $I_k(.)$ is the indicator function of the state $k$. The mean sojourn time up to time $t$ is given by
\begin{align}
\mathbb{E}V_k(t)=\int_0 ^t \Pr \ll L(s)=k |L(0) =1 \rr ds.
\end{align}
By means of \eqref{state prob for L(t)} we have that
\begin{align}
\mathbb{E}V_k(t) \, = \, & \int_0^t \Pr \ll L(s)=k |L(0)=1 \rr ds \notag \\
= \, & \frac{(-1)^{k-1} \lambda^{k-1}}{k!} \frac{d^k}{d\lambda^k} \l \lambda \l t-\int_0^t \Pr \ll L(s)=0\rr ds \r \r \notag  \\
= \, &  \frac{(-1)^{k-1} \lambda^{k-1}}{k!} \frac{d^k}{d\lambda^k} \l \lambda \l t-\int_0^t \frac{\lambda s}{1+\lambda s}ds \r \r \notag  \\
= \, &  \frac{(-1)^{k-1} \lambda^{k-1}}{k!} \frac{d^k}{d\lambda^k}  \log (1+\lambda t) \notag \\
= \, & \frac{1}{\lambda k} \l \frac{\lambda t}{1+\lambda t} \r^k
\end{align}
and the mean asymptotic sojourn time is therefore given by
\begin{align}
\mathbb{E}V_k(\infty)=\frac{1}{\lambda k}. \label{mean classic sojourn time}
\end{align}
 In view of \eqref{state probabilities L^f(t)}, for the  sojourn time $V_k^f(t)$ of the subordinated process $L^f(t)$ we have that
\begin{align}
\mathbb{E}V_k^f(t) \, = \, &\int_0^t \Pr \ll L^f(s)=k |L^f(0) =1 \rr ds \notag \\
= \, & \frac{(-1)^{k-1} \lambda^{k-1}}{k!} \frac{d^k}{d\lambda^k} \left[ \lambda \int_0^\infty dw \, e^{-w} \frac{1}{f(\lambda w)} \l 1- e^{-tf(\lambda w)} \r \right]
\end{align}
and the mean asymptotic sojourn time is given by
\begin{align}
\mathbb{E}V_k^f(\infty)= \frac{(-1)^{k-1} \lambda^{k-1}}{k!} \frac{d^k}{d\lambda^k} \left[ \lambda \int_0^\infty dw \, e^{-w}  \frac{1}{f(\lambda w)} \right].
\end{align}
It is possible to obtain an explicit expression for $   \mathbb{E}V_k^f(\infty)$ in the case of a stable subordinator, when $f(x)=x^{\alpha}$, $\alpha \in (0,1)$, i.e.
\begin{align}
\mathbb{E}V_k^f(\infty) \, = \, & \frac{(-1)^{k-1} \lambda^{k-1}}{k!} \frac{d^k}{d\lambda^k} \left[ \lambda \int_0^\infty dw \, e^{-w}  \frac{1}{\lambda ^{\alpha}w^{\alpha}} \right] \notag \\
= \, & \frac{(-1)^{k-1} \lambda^{k-1}\Gamma (1-\alpha)}{k!}  \frac{d^k}{d \lambda ^k} \lambda ^{1-\alpha} \notag\\
= \, & \frac{(-1)^{k-1} \lambda^{k-1}\Gamma (1-\alpha)}{k!}  (1-\alpha)(-\alpha)(-\alpha-1)\cdots (-\alpha -k+1) \lambda ^{-\alpha -k+1}  \notag\\
 = \, & \frac{\Gamma(1-\alpha)\Gamma(\alpha+k)}{k!\Gamma(\alpha)\lambda^{\alpha}}\notag \\
 = \, & \frac{B(1-\alpha, k+ \alpha)}{\Gamma(\alpha) \lambda^{\alpha}}, \qquad \textrm{for } k \geq 1.
\label{tempo medio di soggiorno stabile}
\end{align}
In the case $\alpha= \frac{1}{2}$, by using the duplication formula for the Gamma function and the Stirling formula, the quantity in \eqref{tempo medio di soggiorno stabile} can be estimated,  for large values of $k$, in the following way:
\begin{align}
\mathbb{E} V_k^f(\infty)=\frac{\Gamma(\frac{1}{2}+k)}{k! \sqrt{\lambda}}  = \frac{\Gamma(\frac{1}{2})2^{1-2k}\Gamma(2k)}{k! \sqrt{\lambda}\Gamma(k)}    \simeq \frac{1}{\sqrt{\lambda k}}
\end{align}
which is somehow related to \eqref{mean classic sojourn time}. We finally note that
\begin{align}
\frac{1}{(\alpha + k) \Gamma(\alpha) \lambda^\alpha} < \mathbb{E} V_k^f(\infty) < \frac{1}{(1-\alpha)\Gamma(\alpha)\lambda^{\alpha}}, \qquad \forall k \geq 1,
\end{align}
since
\begin{align}
\frac{1}{(\alpha +k)}  < B(1-\alpha, k+\alpha)< \frac{1}{1-\alpha}.
\end{align}

\subsection{On the distribution of the sojourn times}
Let $L^f_k(t)$ be a linear birth-death process with $k$ progenitors. We now study the distribution of the sojourn time
\begin{align}
V_{k}(t) \, = \, \int_0^t I_k \l L^f_k(s)  \r  ds
\end{align}
which represents the total amount of time that the process spends in the state $k$ up to time $t$. We now define the Laplace transform
\begin{align}
r_k(\mu) \, = \, \int_0^\infty e^{-\mu t} \Pr \ll L^f_k(t) = k \rr dt.
\end{align}
The hitting time
\begin{align}
V_k^{-1}(t) \, = \, \inf \ll w > 0 : V_k(w) > t \rr
\end{align}
is such that
\begin{align}
\mathbb{E} \int_0^\infty e^{-\mu V_k^{-1}(t)} dt \, = \, & \mathbb{E} \int_0^\infty e^{-\mu t} dV_k(t) \notag \\
= \, & \mathbb{E} \int_0^\infty e^{-\mu t} I_k\l L_k^f(t) \r dt \notag \\
= \, &r_k(\mu).
\label{14}
\end{align}
By Proposition 3.17, chapter V, of \cite{getoor} we have
\begin{align}
\mathbb{E} e^{-\mu V_k^{-1}(t)} \, = \, e^{-t\frac{1}{r_k(\mu)}}.\label{sss}
\end{align}
Now we resort to the fact that
\begin{align}
\Pr \ll V_k(t) > x \rr \, = \, \Pr \ll V_k^{-1}(x) < t \rr
\end{align}
and thus we can write
\begin{align}
\Pr \ll V_k(t) \in dx \rr /dx\, = \, -\frac{\partial}{\partial x} \int_0^t \Pr \ll V_k^{-1}(x) \in dw \rr.
\label{inverso}
\end{align}
We therefore have that
\begin{align}
\frac{1}{dx}\int_0^\infty e^{-\mu t}\Pr \ll V_k(t) \in dx \rr dt \, & = -\frac{d}{dx} \int_0^{\infty} dt \, e^{-\mu t} \int_0^{t} \Pr \ll V_k^{-1}(x) \in dw\rr \notag \\
&= - \frac{d}{dx} \int_0^{\infty} dw \int _w^{\infty} dt \, e^{-\mu t} \Pr \ll V_k^{-1}(x) \in dw\rr \notag \\
&= -\frac{1}{\mu} \frac{d}{dx} \int_0^{\infty} dw\, e^{- \mu w} \Pr \ll V_k^{-1}(x) \in dw\rr \notag \\
  & =-\frac{1}{\mu} \frac{d}{d x}  e^{-x \frac{1}{r_k(\mu)}} \notag \\
&=  \frac{1}{\mu \, r_k(\mu)}e^{-x \frac{1}{r_k(\mu)}}.
\end{align}
If $r_k(0)<\infty$, from \eqref{sss} it emerges that $\Pr \ll V_k^{-1}(t) < \infty \rr<1$; so the sample paths of $V_k(t)$ become constant after a random time with positive probability. This is related to the fact that the subordinated birth and death process extinguishes with probability one  in a finite time when $\lambda = \mu$.

We finally observe that in the case $k=1$  by \eqref{state probabilities L^f(t)} we have
\begin{align}
r_1(\mu) \, = \, \int_0^{\infty} e^{-\mu t}\Pr \ll L^f(t)=k \rr dt \, = \,  \frac{d}{d\lambda}\left[ \lambda \int_0^{\infty}dw \, e^{-w}\frac{1}{\mu+f(\lambda w)}\right],
\end{align}
provided that the Fubini Theorem holds true.

\vspace{1cm}
\end{document}